\documentclass[11pt,letterpaper]{article}

\usepackage{amsmath,amsthm,amsfonts,amssymb,bm,wasysym}
\usepackage{subfigure}
\usepackage{graphicx}
\usepackage[usenames]{color}
\usepackage{verbatim}
\usepackage{epsfig}

\usepackage{float}

\usepackage{fullpage}
\usepackage{paralist}
\usepackage{hyperref}

\newtheorem{theorem}{Theorem}[section]
\newtheorem{definition}[theorem]{Definition}

\numberwithin{equation}{section}
\newtheorem{lemma}[theorem]{Lemma}
\newtheorem{proposition}[theorem]{Proposition}

\numberwithin{equation}{section}

\renewcommand{\phi}{\varphi}
\renewcommand{\epsilon}{\varepsilon}

\renewcommand{\emptyset}{\varnothing}

\newcommand{\tn}{|\kern-.1em|\kern-0.1em|}

\newcommand\be{\begin{equation}}
\newcommand\ee{\end{equation}}

\newcommand{\citet}{\cite}
\newcommand{\citep}{\cite}

\usepackage{float}
\author{
{\sf Christian Borgs}\thanks{Microsoft Research New England e-mail: {\tt Christian.Borgs@microsoft.com}. }
\and
{\sf Jennifer Chayes}\thanks{Microsoft Research New England e-mail: {\tt jchayes@microsoft.com}.}
\and
{\sf Adam Smith}\thanks{Boston University; e-mail: {\tt ads22@bu.edu}. }
\and
{\sf Ilias Zadik}\thanks{MIT; e-mail: {\tt izadik@mit.edu}. Research done in part while an intern at Microsoft Research New England. }
}

\begin{document}

\title{Private Algorithms Can Always Be Extended}
\date{\today}

\maketitle

\begin{abstract}
We consider the following fundamental question on $\epsilon$-differential privacy. Consider an arbitrary $\epsilon$-differentially private algorithm defined on a subset of the input space. Is it possible to extend it to an $\epsilon'$-differentially private algorithm on the whole input space for some $\epsilon'$ comparable with $\epsilon$? In this note we answer affirmatively this question for $\epsilon'=2\epsilon$. Our result applies to every input metric space and space of possible outputs. This result originally appeared in a recent paper by the authors~\cite{BCSZ18}. We present a self-contained version in this note, in the hopes that it will be broadly useful. 
\end{abstract}

\section{Introduction}

We are undoubtedly living in a revolutionary era for data science. The number and magnitude of the available datasets have been growning enormously during the last years. Arguably, though, what makes the data so useful is also what makes them very sensitive. In order to get a theoretical handle on the extent to which an algorithm reveals too much about an individual's data, Dwork et al.~\cite{DMNS06} introduced \emph{differential privacy}, a mathematical property of statistical algorithms which guarantees privacy of individuals' input data.  Roughly speaking, differential privacy requires that a change to one individual's input data not affect the algorithm's output distribution too much.

Differential privacy appears in the literature in various forms. It is typically presented in its ``relaxed form" with respect to two parameters $\epsilon>0$ and $\delta \geq 0$,~\cite{DMNS06}. In this note, we focus solely on the special case $\delta=0$, sometimes called ``pure'' differential privacy, or just $\epsilon$-differential privacy. Differential privacy is a property of a randomised algorithm which takes values on some metric space $(\mathcal{M},d)$ and outputs distributions on some measurable space $(\Omega,\mathcal{F})$:

\begin{definition}\label{dfn}
A randomized algorithm $\mathcal{A}$ is $\epsilon$-differential private if for all subsets $S \in \mathcal{F}$ of the output measurable space $(\Omega, \mathcal{F})$ and data-sets $D_1,D_2$ of the input metric space $(\mathcal{M},d)$,  \begin{equation} \label{privdfn}\mathbb{P}\left(\mathcal{A}(D_1) \in S\right) \leq \exp\left[\epsilon d(D_1,D_2)\right]\mathbb{P}\left(\mathcal{A}(D_2) \in S\right).\end{equation}
\end{definition} The parameter $\epsilon>0$ should be treated as a privacy measure; the smaller the $\epsilon$ the higher the privacy guarantee.

An intriguing characteristic of analyzing the performance of an $\epsilon$-differential private algorithms is an inherent trade-off between accuracy and privacy. This is simply the property that increasing the required privacy level of the algorithm (decreasing $\epsilon>0$) constrains the accuracy levels of estimation. The natural quantitiative questions rises: how much accuracy is necessarily sacrificed if we entitle our algorithms to a minimum level of privacy? To make the question more precise, suppose that we face an arbitrary statistical estimation question and we restrict ourselves to the use of algorithms of a certain differential privacy level $\epsilon$. What is the optimal accuracy that can be achieved? How does it compare to the optimal accuracy achieved without any privacy guarantee?  Despite the simplicity of this question, in most examples the exact calculation of the underlying rates of estimation remain, to the best of our knowledge, largely open.

A key feature of any estimation question is the generating data distribution. Notably, this distribution does not affect at all the $\epsilon$-differential privacy constraint, as (\ref{privdfn}) should hold for any pair of input datasets. On the other hand, the assumed generating distribution can massively change the accuracy guarantee. For example, let us consider the following problem, studied in~\cite{KNRS13,ChenZ13,BBDS13}; the analyst receives a sampled graphs $G$ from a distribution over bounded degree undirected graphs on $n$ vertices. The goal is to estimate, in the mean squared error sense, the edge density of the input graph using an $\epsilon$-differentially private algorithm. As we mentioned above, the algorithm needs to satisfy (\ref{privdfn}) for all pairs of graphs, but it needs to be accurate only on graphs appearing with non-negligible probability as input. In particular, our algorithms suffices to accurately estimate the edge density solely for bounded degree graphs, while for the rest graphs it could potentially be extremely inaccurate. If one knows more about how the graph was generated, then one can further restrict the ``interesting'' set of inputs---for example, a graph generated from $G(n,p)$ will have cut density approximately $p$ for all cuts, with high probability (for $p$ not too close to zero). 

Building on this inherent feature of the problem the following strategy for designing differentially private algorithms has appeared in various forms. 
\begin{itemize} 
\item First design an algorithm $\mathcal{A}_1$ that is both accurate and $\epsilon$-differentially private on a set $\mathcal{H} \subset \mathcal{M}$ of ``typical" elements of the generating data distribution. 
\item Extend the algorithm $\mathcal{A}_1$ to a private algorithm $\mathcal{A}_2$ on the whole space of inputs $\mathcal{M}$ so that (a) it is $\epsilon'$-differentially private on the whole space for some $\epsilon'  \geq  \epsilon$ and (b) when the input belongs in $\mathcal{H}$, it outputs the same distribution as the original algorithm $\mathcal{A}_1$. \end{itemize} Various such extension results has appeared in the literature~\cite{KNRS13,BBDS13,RaskhodnikovaS16,DayLL16,CummingsD18} but they are usually tailored for the specific applications each paper discusses. Here, we give a simple proof that such an extension is \textbf{always possible} for $\epsilon'=2\epsilon$. 

Differential privacy, as defined in Definition \ref{dfn}, can be easily shown to be equivalent with the existence of an $\epsilon$-Lipschitz map from the space of input data $(\mathcal{M},d)$ to the space of probability measures of some sample space $(\Omega,\mathcal{F})$ endowed with the $\infty$-Renyi divergence metric, $D_{\infty}(\mu,\nu)=\log \sup_{S \in \mathcal{F}}| \frac{ \mu(S)}{\nu(S)}|.$ Indeed, the condition (\ref{privdfn}) applied to two datasets $D_1,D_2$ is equivalent with the condition $D_{\infty}( \mathcal{A}(D_1),\mathcal{A}(D_2) ) \leq \epsilon d(D_1,D_2).$

Viewed from this perspective, the extension of an $\epsilon$-differentially private mechanism reduces to a Lipschitz extension question, where a function is assumed to be Lipschitz on a subset of the domain, and the goal is to be extented to a Lipschitz function on the whole domain. The extendability of Lipschitz functions has been thoroughly studied in the field of functional analysis, see for example the Lecture Notes~\cite{NaorLN}. One standard result in this line of research, states that if the image space is an $\ell_{\infty}(\Gamma)$ space for some set $\Gamma$, such an extension is always possible with $\epsilon'=\epsilon$ (see Theorem 2.1 in~\cite{NaorLN}). Unfortunately the image of differentiable private algorithms does not have the exact structure of an $\ell_{\infty}$ space. Despite that, using similar ideas with one of the standard proofs with the $\ell_{\infty}$ result, we are able to establish the general extension result for $\epsilon'=2\epsilon$.

\section{The Extension Result}

We consider the following statistical model.

\subsection*{The Model}
Let $n \in \mathbb{N}$ and $\epsilon>0$. We assume that the analyst's objective is to estimate a certain quantity which belongs in some measurable space $(\Omega,\mathcal{F})$ from input data which takes values in a metric space $(\mathcal{M},d)$. The analyst is assumed to use for this task a randomized algorithm $\mathcal{A}$ which should be \begin{itemize}
\item[(1)] as highly \textbf{accurate} as possible for input data belonging in some \textit{hypothesis set} $\mathcal{H} \subseteq \mathcal{M}$;

\item[(2)] $\epsilon$-\textbf{differentially private} for arbitrary pairs of input data-sets from $(\mathcal{M},d)$.
\end{itemize}

\subsection*{Extending Private Algorithms}

We now state formally the result described in the note. Consider an arbitrary $\epsilon$-differentially private algorithm defined on input belonging in some set $\mathcal{H} \subset \mathcal{M}.$ We show that it \textbf{can be always extended} to a $2\epsilon$-differentially private algorithm defined for arbitrary input data from $\mathcal{M}$ with the property that if the input data belongs in $\mathcal{H}$, the distribution of output values is the same with the original algorithm. We state formally the result. 

\begin{proposition}[``Extending Private Algorithms at $\epsilon$-cost"]\label{extension}
Let $\hat{\mathcal{A}}$ be an $\epsilon$-differentially private algorithm designed for input from $\mathcal{H} \subseteq \mathcal{M}$. Then there exists a randomized algorithm $\mathcal{A}$ defined on the whole input space $\mathcal{M}$ which is $2\epsilon$-differentially private and satisfies that for every $D \in \mathcal{H}$, $\mathcal{A}(D) \overset{d}{=}  \hat{\mathcal{A}}(D)$.
\end{proposition}

\section{Applications}

In this section, we describe two applications of Theorem \ref{extension}.

\subsection*{Bounded Degree Graphs}
Let $D,n \in \mathbb{N}$ with $D \leq n$.  The authors in~\cite{KNRS13}  discuss the following model. Say that an analyst observes a network and wants to build a private algorithm for estimating a graph quantity of the network (e.g. the number of triangles, the degree histogram or the number of edges). Differential privacy here is defined with respect to the node or rewiring distance; the exact definition of this distance is not in the scope of the note but we encourage the interested reader to the discussion in Section 2.1. of ~\cite{KNRS13}, or Section 1 in  ~\cite{BCSZ18}. Now motivated by the apparent sparsity of various real life networks, the authors in~\cite{KNRS13} face the following question; consider an algorithm which is $\epsilon$-differentially private over the space of undirected graphs on $n$ vertices with maximum degree at most $D$, which contains all ``realistic" graphs for relatively small values of $D$.  Can we extend this algorithm to a differentially private algorithm on the whole space of undirected graphs on $n$ vertices? 

Using an efficient approach which is tailored for bounded degree graphs, they establish the existence of an extension algorithm which is $(2\epsilon,\delta)$-differentially private for some $\delta \geq 0$ (Lemma 6.2 in~\cite{KNRS13}).

Notice, though, that the question falls exactly into the setting of our note. In particular, for $\mathcal{M}$ the space of undirected graphs on $n$ vertices endowed with the node distance and $\mathcal{H}$ the set of graphs with bounded degree $D$, Theorem \ref{extension} implies the existence of an $2\epsilon=(2\epsilon,0)$-differentially private extension. Therefore, we conclude the existential aspect of Lemma 6.2. in~\cite{KNRS13} as a special case of our result, where we improve the $\delta$ to be equal to 0. It is important, though, to notice that unlike the results in~\cite{KNRS13}, our result does not have any efficiency guarantee.

\subsection*{Estimating Random Graphs} 

Theorem \ref{extension} firstly appeared in~\cite{BCSZ18}, where the authors discuss the following fundamental and to the best of our knowledge unexplored question: suppose we receive a sample from a $G(n,p)$ Erdos Renyi random graph model, where $n$ is known and $p$ is unknown. How well can we estimate $p$, in the mean squared error sense, using an $\epsilon$-differentially private algorithm? Note that differential privacy is again understood here with the respect to the node distance. The standard approach for such a question is to add appropriate Laplace noise to the edge density of the observed graph. In~\cite{BCSZ18} it is established that all such estimators can imply at best a rate of the order $$\frac{1}{n^2}+\frac{1}{n^2 \epsilon^2}.$$However, it turns out that this rate is suboptimal. The main reason is that for any level of Laplace noise, the estimator can not take into account the ``typical" homogeneous structure of an Erdos Renyi graph. In our paper~\cite{BCSZ18} we take advantage of this property and construct a set $\mathcal{H}$ which captures a typical homogeneous structure of an Erdos Renyi graph. We first build an algorithm on graphs belonging in $\mathcal{H}$, apply Theorem \ref{extension} to extend the algorithm and finally obtain a rate$$\frac{1}{n^2}+\frac{\log n}{n^3 \epsilon^2}.$$The authors in~\cite{BCSZ18} provide also a tight lower bound result (Theorem 4.5. in ~\cite{BCSZ18}) in the similar case where the graph is sampled from the uniform model $G(n,m)$ and $\epsilon$ is constant with respect to $n$. The lower bound result proves that the rate obtained using Theorem \ref{extension} is optimal for the $G(n,m)$ case and suggests the same for $G(n,p)$. 

%
%

\section{Proof of Theorem \ref{extension}}\label{ext}
  We start with a lemma.
\begin{lemma}\label{lem}
Let $\mu$ be a probability measure on $\Omega$ and $\mathcal{A}'$ be a randomized algorithm designed for input from $\mathcal{H}' \subseteq \mathcal{M}$. Suppose that for any $D \in \mathcal{H}'$, $\mathcal{A}'(D)$ is absolutely continuous to $\mu$ and let $f_D$ the Radon-Nikodym derivative $\frac{d\mathcal{A}'(D)}{d\mu}$. Then the following are equivalent \begin{itemize} \item[(1)] $\mathcal{A}'$ is $\epsilon$-differentially private; \item[(2)] For any $D,D' \in \mathcal{H}$ \begin{equation}\label{prime}  f_{\mathcal{A}'(D)} \leq \exp \left(\epsilon d(D,D') \right) f_{\mathcal{A}'(D')}, \end{equation}$\mu$-almost surely. 
\end{itemize}
\end{lemma}

\begin{proof}
For the one direction, suppose $\mathcal{A}'$ satisfies (\ref{prime}). Then for any set $S \in \mathcal{F}$ we obtain
 \begin{align*} \mathbb{P}\left(\mathcal{A}'(D) \in S \right) &=\int_{S}  f_{\mathcal{A}'(D)} d\mu \\
& \leq \exp \left(\epsilon d(D,D') \right) \int_{ S} f_{\mathcal{A}'(D')} d\mu\\
&= \exp \left(\epsilon d(D,D') \right)\mathbb{P}\left(\mathcal{A}'(D) \in S \right).
 \end{align*}We prove the other direction by contradiction. Consider the set $$S= \{ f_{\mathcal{A}'(D)} > \exp \left(\epsilon d(D,D') \right) f_{\mathcal{A}'(D')} \} \in \mathcal{F}$$ and assume that $\mu (S)>0$. By definition on being strictly positive on a set of positive measure \begin{equation*}\int_S \left[f_{\mathcal{A'}(D)} -\exp \left(\epsilon d(D,D') \right) f_{\mathcal{A}'(D')}\right]  d \mu>0 \end{equation*}or equivalently\begin{equation}\label{contrad}\int_S f_{\mathcal{A'}(D)}d \mu > \exp \left(\epsilon d(D,D') \right) \int_Sf_{\mathcal{A}'(D')}  d \mu. \end{equation} On the other hand using $\epsilon$-differential privacy we obtain
\begin{align*}  \int_S f_{\mathcal{A'}(D)}d \mu & =\mathbb{P}(\mathcal{A'}(D) \in S) \\\
&\leq \exp \left(\epsilon d(D,D') \right)\mathbb{P}(\mathcal{A'}(D') \in S) \\
&= \exp \left(\epsilon d(D,D') \right) \int_Sf_{\mathcal{A}'(D')}  d \mu,
\end{align*} a contradiction with (\ref{contrad}). This completes the proof of the Lemma.
\end{proof}
Now we establish Theorem \ref{extension}.

\begin{proof}

Since $\mathcal{H} \not  = \emptyset$, let $D_0 \in \mathcal{H}$ and denote by $\mu$ the measure $\hat{\mathcal{A}}(D_0)$. From the definition of differential privacy we know for all $D \in \mathcal{H}$ and $S \in \mathcal{F}$, if $\mathbb{P}\left( \hat{\mathcal{A}}(D_0) \in S\right) =0$ then $\mathbb{P}\left( \hat{\mathcal{A}}(D) \in S\right) =0$. In the language of measure theory that means the measure $\hat{\mathcal{A}}(D)$ is absolutely continuous to $\mathcal{A}(D_0)$. By Radon-Nikodym theorem we conclude that there are measurable functions $f_D: \Omega \rightarrow [0,+\infty)$ such that for all $S \in \mathcal{F}$, \begin{equation}\mathbb{P}\left( \hat{\mathcal{A}}(D) \in S\right) = \int_S f_D d\mu. \end{equation}

 We define now the following randomized algorithm $\mathcal{A}$. For every $D \in \mathcal{M}$, $\mathcal{A}(D)$ samples from $\Omega$ according to the absolutely continuous to $\mu$ distribution with density proportional to $$\inf_{D' \in \mathcal{H}} \left[ \exp\left(\epsilon d(D,D')\right) f_{\hat{\mathcal{A}}(D')} \right].$$ That is for every $\omega \in \Omega$ its density with respect to $\mu$  is defined as \begin{equation*}f_{\mathcal{A}(D)}(\omega) =\frac{1}{Z_D} \inf_{D' \in \mathcal{H}} \left[ \exp\left(\epsilon d(D,D')\right) f_{\hat{\mathcal{A}}(D')}(\omega) \right],\end{equation*} where \begin{equation*}Z_D:=\int_{\Omega} \inf_{D' \in \mathcal{H}} \left[ \left(\epsilon d(D,D')\right) f_{\hat{\mathcal{A}}(D)'} \right] d \mu.\end{equation*}In particular for all $S \in \mathcal{F}$ it holds $$\mathbb{P}(\mathcal{A}(D) \in S)=\int_S f_{\mathcal{A}(D)}d \mu.$$

We first prove that $\mathcal{A}$ is $2\epsilon$-differentially private over all pairs of input from $\mathcal{M}$. Using Lemma \ref{lem} it suffices to prove that for any $D_1,D_2 \in \mathcal{H}$,
 \begin{align*}
f_{\mathcal{A}(D_1)} \leq \exp \left(2 \epsilon d(D_1,D_2)\right) f_{\mathcal{A}(D_2)},
\end{align*} $\mu$-almost surely. We establish it in particular for every $\omega \in \Omega$. Let $D_1,D_2 \in \mathcal{M}$. Using triangle inequality we obtain for every $\omega \in \Omega$, \begin{align*}\inf_{D' \in \mathcal{H}} \left[ \exp \left( \epsilon d(D_1,D')\right) f_{\hat{\mathcal{A}}(D')}(\omega) \right] & \leq \inf_{D' \in \mathcal{H}} \left[ \exp\left(\epsilon \left[d(D_1,D_2)+d(D_2,D')\right] \right) f_{\hat{\mathcal{A}}(D')}(\omega) \right]\\
&=\exp \left( \epsilon d(D_1,D_2)\right) \inf_{D' \in \mathcal{H}} \left[ \exp \left(\epsilon d(D,D')\right) f_{\hat{\mathcal{A}}(D')}(\omega) \right],
\end{align*}which implies that for any $D_1,D_2 \in \mathcal{M}$, 
\begin{align*}
Z_{D_1}&=\int_{\Omega} \inf_{D' \in \mathcal{H}} \left[ \exp\left(\epsilon d(D_1,D')\right) f_{\hat{\mathcal{A}}(D')} \right] d \mu \\
&\leq \exp\left(\epsilon d(D_1,D_2) \right)\int_{\Omega} \inf_{D' \in \mathcal{H}} \left[ \exp \left( \epsilon d(D_2,D') \right) f_{\hat{\mathcal{A}}(D')}(\omega) \right] d\mu\\
&=\exp \left( \epsilon d(D_1,D_2) \right) Z_{D_2}. 
\end{align*}Therefore using the above two inequalities we obtain that for any $D_1,D_2 \in \mathcal{H}$ and $\omega \in \Omega$,
 \begin{align*}
f_{\mathcal{A}(D_1)}(\omega) &=\frac{1}{Z_{D_1}} \inf_{D' \in \mathcal{H}} \left[ \exp \left( \epsilon d(D_1,D') \right) f_{\hat{\mathcal{A}}(D')}(\omega) \right] \\
&\leq \frac{1}{\exp\left(-\epsilon d(D_2,D_1)\right)Z_{D_2}}\exp\left(\epsilon d(D_1,D_2)\right) \inf_{D' \in \mathcal{H}}  \left[ \exp \left( \epsilon d(D_2,D') \right) f_{\hat{\mathcal{A}}(D')}(\omega) \right] \\
&=\exp \left(2 \epsilon d(D_1,D_2)\right) \frac{1}{Z_{D_2}} \inf_{D' \in \mathcal{H}}  \left[ \exp \left( \epsilon d(D_2,D') \right) f_{\hat{\mathcal{A}}(D')}(\omega) \right]\\
&=\exp \left(2 \epsilon d(D_1,D_2)\right) f_{\mathcal{A}(D_2)}(\omega),
\end{align*}as we wanted. 

Now we prove that for every $D \in \mathcal{H}$, $\mathcal{A}(D) \overset{d}{=}  \hat{\mathcal{A}}(D)$. Consider an arbitrary $D \in \mathcal{H}$. From Lemma \ref{lem} we obtain that $\hat{\mathcal{A}}$ is $\epsilon$-differentially private which implies that for any $D,D' \in \mathcal{H}$ \begin{equation}  f_{\hat{\mathcal{A}}(D)} \leq \exp \left(\epsilon d(D,D') \right) f_{\hat{\mathcal{A}}(D')},   \end{equation} $\mu$-almost surely. Observing that the above inequality holds as $\mu$-almost sure equality if $D'=D$ we obtain that for any $D \in \mathcal{H}$ it holds $$f_{\hat{\mathcal{A}}(D)}(x)=\inf_{D' \in \mathcal{H}}  \left[ \exp \left( \epsilon d(D,D') \right) f_{\hat{\mathcal{A}}(D')}(x) \right],$$ $\mu$-almost surely. Using that $f_{\hat{\mathcal{A}}(D)}$ is the Radon-Nikodym derivative $\frac{d \hat{\mathcal{A}}(D)}{d \mu}$  we conclude $$Z_{D}:=\int_{\Omega} f_{\hat{\mathcal{A}}(D)} d\mu=\mu(\Omega)=1.$$ Therefore $$f_{\hat{\mathcal{A}}(D)}=\frac{1}{Z_D}\inf_{D' \in \mathcal{H}}  \left[ \exp \left( \epsilon d(D,D') \right) f_{\hat{\mathcal{A}}(D')} \right],$$ $\mu$-almost surely and hence $$ f_{\hat{\mathcal{A}}(D)}=f_{\mathcal{A}(D)},$$ $\mu$-almost surely. This suffices to conclude that $\hat{\mathcal{A}}(D) \overset{d}{=}  \mathcal{A}(D)$ as needed. 

The proof of Theorem \ref{extension} is complete.
\end{proof}

\section{An Open Problem: Efficiency}Theorem \ref{extension} answers affirmatively the question of extendability of an arbitrary $\epsilon$-differential private algorithm.  An important and interesting open problem for future work is the underlying computational question; under which conditions such an extension can become computationally efficient? This was the focus of much of the existing work~\cite{KNRS13,BBDS13,RaskhodnikovaS16,DayLL16,CummingsD18}; our general result suggests that much greater generality is possible for polynomial-time extensions.

\bibliographystyle{plain}
\bibliography{forArxiv,Lipschitz-extensions}

\begin{thebibliography}{1}

\bibitem{BBDS13}
Jeremiah Blocki, Avrim Blum, Anupam Datta, and Or~Sheffet.
\newblock Differentially private data analysis of social networks via
  restricted sensitivity.
\newblock In {\em Innovations in Theoretical Computer Science {(ITCS)}}, pages
  87--96, 2013.

\bibitem{BCSZ18}
Christian Borgs, Jennifer~T. Chayes, Adam~D. Smith, and Ilias Zadik.
\newblock Revealing network structure confidentially: Improved rates for
  node-private graphon estimation.
\newblock In {\em Symposium of the Foundations of Computer Science (FOCS)},
  2018.

\bibitem{ChenZ13}
Shixi Chen and Shuigeng Zhou.
\newblock Recursive mechanism: towards node differential privacy and
  unrestricted joins.
\newblock In {\em {ACM} {SIGMOD} International Conference on Management of
  Data}, pages 653--664, 2013.

\bibitem{CummingsD18}
Rachel Cummings and David Durfee.
\newblock Individual sensitivity preprocessing for data privacy.
\newblock {\em CoRR}, abs/1804.08645, 2018.

\bibitem{DayLL16}
Wei{-}Yen Day, Ninghui Li, and Min Lyu.
\newblock Publishing graph degree distribution with node differential privacy.
\newblock In {\em International Conference on Management of Data {SIGMOD}},
  pages 123--138, 2016.

\bibitem{DMNS06}
Cynthia Dwork, Frank McSherry, Kobbi Nissim, and Adam Smith.
\newblock Calibrating noise to sensitivity in private data analysis.
\newblock In {\em Theory of Cryptography Conference (TCC)}, pages 265--284,
  2006.

\bibitem{KNRS13}
Shiva~Prasad Kasiviswanathan, Kobbi Nissim, Sofya Raskhodnikova, and Adam
  Smith.
\newblock Analyzing graphs with node-differential privacy.
\newblock In {\em Theory of Cryptography Conference (TCC)}, pages 457--476,
  2013.

\bibitem{NaorLN}
Assaf Naor.
\newblock Metric embeddings and lipschitz extensions, lecture notes.
\newblock 2015.

\bibitem{RaskhodnikovaS16}
Sofya Raskhodnikova and Adam~D. Smith.
\newblock Lipschitz extensions for node-private graph statistics and the
  generalized exponential mechanism.
\newblock In {\em Symposium on Foundations of Computer Science ({FOCS})}, pages
  495--504, 2016.

\end{thebibliography}

\end{document}